\documentclass[a4paper,11pt]{amsproc}

\usepackage[english]{babel}
\usepackage[utf8x]{inputenc}
\usepackage{geometry} 
\usepackage{euscript}

\usepackage[pagebackref,colorlinks,linkcolor=blue,citecolor=blue,urlcolor=blue,hypertexnames=true]{hyperref}
\usepackage{amsrefs}

\usepackage{graphicx} 
\usepackage{epstopdf}
\usepackage{amscd}
\usepackage{array} 
\usepackage{verbatim} 
\usepackage{amssymb}
\usepackage{amsmath}
\usepackage{amsthm}

\newtheorem{defin}{Definition}

\newtheorem{theorem}{Theorem}
\newtheorem{lemma}{Lemma}
\newtheorem{prop}{Proposition}

\newtheorem{question}{Question}
\newtheorem{conjecture}{Conjecture}

\title{On the isometrisability of group actions on $p$-spaces}

\author[M. Gerasimova]{Maria Gerasimova}
\address{Maria Gerasimova, TU Dresden, Germany}
\email{maria.gerasimova@tu-dresden.de}

\author[A. Thom]{Andreas Thom}
\address{Andreas Thom, TU Dresden, Germany}
\email{andreas.thom@tu-dresden.de}

\begin{document}
\maketitle

\begin{abstract}
In this note we consider a $p$-isometrisability property of discrete groups. If $p=2$ this property is equivalent to unitarisability. We prove that any group containing a non-abelian free subgroup is not $p$-isometrisable for any $p\in (1, \infty)$. We also discuss some open questions and possible relations of $p$-isometrisability with the recently introduced Littlewood exponent ${\rm Lit}(\Gamma)$.
\end{abstract}

\tableofcontents

\section{Introduction}
Let $\Gamma$ be a discrete group. For a Banach space $X$, we write ${\rm GL}(X)$ for the group of bounded invertible operators on $X$. A representation $\pi \colon \Gamma \to {\rm GL}(X)$ of a group $\Gamma$ is called \textit{uniformly bounded} if 
$$\sup \limits_{g\in \Gamma} \|\pi(g)\| <\infty.$$

In case when $X$ is a Hilbert space $\mathcal{H}$, the study of uniformly bounded representations has a long history dating back to \cite{Dixmier}. We say that a representation $\pi \colon \Gamma \to {\rm GL}(\mathcal{H})$ is called \textit{unitarisable}, if there exists $S \in {\rm GL}(\mathcal{H})$ such that $S\pi(g)S^{-1}$ is a unitary operator for every $g\in \Gamma$, i.e. $\pi$ is conjugated to a unitary representation. Equivalently, we have that $\pi$ is unitarisable if there exists a equivalent scalar product, which makes $\pi$ unitary. It is clear that any representation that is conjugated to a unitary representation is uniformly bounded. A group $\Gamma$ is said to be \textit{unitarisable} if any uniformly bounded representation is unitarisable. It is well known that amenable groups are unitarisable \cite{Day, Dixmier, Nakamura}. The following question is still open:
\begin{question}[Dixmier]
Are unitarisable groups amenable?
\end{question}

In this short note, we study a generalization of the notion of unitarisability in the following way. First of all, for $p \in (1,\infty)$, a natural class of Banach spaces called $p$-spaces can be introduced and a Banach space is a $2$-space if and only it is isometrically isomorphic to a Hilbert space. We will say that a group $\Gamma$ is \textit{$p$-isometrisable} if for any uniformly bounded representation $\pi \colon \Gamma \to {\rm GL}(X)$ on a $p$-space $X$ there exists a new $p$-norm which makes $\pi$ isometric. 

Our first result is the following theorem.

\begin{theorem}
Amenable groups are $p$-isometrisable for any $p>1$.
\end{theorem}
Note that the similar result was proven in \cite{arch} using different methods.\\
The second natural question, arises in this context, if there exist non-isometrisable groups for some $p$ (or for any $p>1$). We prove the following.
\begin{theorem}
For any $p\in (1,\infty)$, any group containing a non-abelian free subgroup is not $p$-isometrisable.
\end{theorem}

Our initial hope that was that there is a relationship between the $p$-isometrisability of a group $\Gamma$ and the recently introduced Littlewood exponent ${\rm Lit}(\Gamma)$, see \cite{GGMT}, however we are not there yet. We discuss some open problems in Section \ref{final}.

\section{Basic properties of $p$-spaces}
Let $p\in [1,\infty)$. Let's start by recalling some basic definitions in order to keep the paper self-contained.
\begin{defin}
A Banach space is called an ${\rm L}^p$-space if it is of the form $L^p(X,\mu)$ for some measure space $X$. A Banach space is called ${\rm QSL}_p$-space (or just a $p$-space for short) if it is isometrically isomorphic to a quotient of a closed subspace of an $L^p$-space.
\end{defin}
Let $M\in M_{n\times m}(\mathbb{C})$. By $\|M\|_{p\to q}$ we will denote the norm of $M$ as a linear operator from $m$-dimensional $\ell^p$-space to $n$-dimensional $\ell^q$-space. It was proven in \cite{Kwapien} that the class of $p$-space can be characterized in the following intrinsic way.
\begin{defin}
A Banach space $X$ is a $p$-space if for any $n\in \mathbb{N}$, the natural inclusion
$$M_n(\mathbb C) \to L(\ell^p(n,X),\ell^p(n,X))$$ is a contraction. In more conrete terms, we have
$$\sum \limits_{i=1}^{n}\left\Vert\sum \limits_{j=1}^n M_{ij}x_j\right\Vert^p_X \leq \sum \limits_{k=1}^n \left \Vert x_k\right\Vert^p_X$$
for any $n$-tuple $(x_1, \ldots, x_n) \in X^n$ and any matrix $M \in M_n(\mathbb{C})$ such that
$$\sum \limits_{i=1}^n \left|\sum \limits_{j=1}^n M_{ij}r_j\right|^p \leq \sum \limits_{k=1}^n |r_k|^p$$ for any $(r_1, \ldots, r_k) \in \mathbb{C}^n$.
\end{defin}

Let us recall some  known properties of $p$-spaces, see for example \cite{Runde} for proofs and or references. It is quite clear that every Banach space is a $1$-space and we already mentioned that every $2$-space is isometrically isomorphic to a Hilbert space.

\begin{lemma} \label{spaces} Let $p,q \in (1,\infty)$.
\begin{enumerate}
\item A closed subspace of a $p$-space is a $p$-space.
\item A quotient of a $p$-space is a $p$-space.
\item The dual of a $p$-space is a $q$-space, where $\frac{1}{p}+\frac{1}{q}=1$.
\item Every $p$-space is reflexive. \label{spaces2}
\item If $A$ and $B$ are $p$-spaces, so is $A\oplus_p B$, the completion of $A\oplus B$ for the norm $$\|(a,b)\|=(\|a\|_A^p+\|b\|_B^p)^{\frac{1}{p}}.$$
\item An $q$-space is a $p$-space if  $2\leq q\leq p$ or $p\leq q \leq 2$.
\end{enumerate}
\end{lemma}

\section{Isometrisability of group actions on $p$-spaces}
Throughout the rest of the paper, let $\Gamma$ be a discrete group.
\begin{defin} Let $p,q \in [1,\infty)$. A uniformly bounded representation $\pi \colon \Gamma \to {\rm GL}(X)$ on a $p$-space $X$ is called $q$-isometrisable, if there exists an equivalent $q$-norm on $X$ such that this representation is isometric with respect to this norm. A group $\Gamma$ is said to be $(p,q)$-isometrisable if for any $p$-space $X$, every uniformly bounded representation $\pi\colon \Gamma \to {\rm GL}(X)$ is $q$-isometrisable.
\end{defin}

Note that the definition allows for examples only if the $p$-space in question is a $q$-space at all -- however, this happens for example when $1 \leq q \leq p \leq 2$ and $2 \leq p \leq q \leq \infty$.
 
The first observation is that every group $\Gamma$ is $(p,1)$-isometrisable. Indeed, for any uniformly bounded representation $\pi$ we can define an equivalent and invariant norm by the formula
$$\|x\|_{\rm new}=\sup_g \|\pi(g)x\|.$$ A second observation is that $(2,2)$-isometrisability is exactly the usual unitarisability.

As mentioned above, $(p,q)$-isometrisability makes only sense when $1\leq q\leq p \leq 2$ or $2 \leq p \leq q < \infty$. The following lemma shows that we may restrict ourselves to the first case by duality.

\begin{lemma} Let $1<p,q<\infty$.
A group $\Gamma$ is $(p,q)$-isometrisable if and only if it is $(p',q')$-isometrisable, where $p',q'$ are the conjugate exponents. 
\end{lemma}

In view of Lemma \ref{spaces} it is clear that for $1 \leq p \leq 2$, $(p,q)$-isometrisability implies $(p,r)$-isometrisability for every $1 \leq r \leq q$. We say that a group is $p$-isometrisable when it is $(p,p)$-isometrisable.

\begin{theorem}
Amenable groups are $p$-isometrisable for any $p\in (1,\infty)$.
\end{theorem}
\begin{proof}
Let $\Gamma$ be an amenable group and let $m\colon \ell^{\infty}(\Gamma) \to \mathbb{C}$ be an left-invariant mean. Let $\pi\colon \Gamma\to {\rm GL}(X)$ be a uniformly bounded representation of $\Gamma$ on a $p$-space $X$. Since $\pi$ is uniformly bounded, then for any fixed $x\in X$ we have $\|\pi(g)x\|^p \leq \|\pi(g)\|^p\|x\|^p$ and hence $\left(g\mapsto \|\pi(g)x\|^p\right) \in \ell^{\infty}(\Gamma)$. Thus we can  define a new norm by the formula:
$$\|x\|_{\rm new}=\sqrt[p]{m\left(g\mapsto \left\Vert\pi(g)x\right\Vert^p\right)}.$$

Clearly it is a norm and easily seen to be equivalent to the original norm. Moreover, we prove that it is a $p$-norm. Indeed, finite additivity and positivity of $m$ implies
\begin{eqnarray*}
\sum_{i=1}^n m\left(\left\Vert\sum_{j=1}^n \pi(g)M_{ij}b_j\right\Vert ^p\right) 
&=&m\left(\sum_{i=1}^n\left(\left\Vert\sum_{j=1}^n \pi(g)M_{ij}b_j\right\Vert^p\right)\right)\\
&\leq& m\left(\sum_{k=1}^n \left\Vert\pi(g)b_k\right\Vert^p\right)\\
&=&\sum_{k=1}^n \left\Vert\pi(g)b_k\right\Vert^p_{\rm new}.
\end{eqnarray*}

Finally, the invariance of $m$ easily implies that $\pi$ is isometric with respect to this norm.
\end{proof}

One could ask if in the definition of $p$-isometrisability we can replace finding a new $p$-norm by finding a bounded invertible operator $S$ such that $S^{-1}\pi(g)S$ is isometric for any $g\in \Gamma$ and indeed, if $p=2$ these conditions are equivalent. However, this is not true for $p\neq 2$. It was proven in \cite{Gillaspie}, that there exists a uniformly bounded representation of $\mathbb{Z}$ on $L^p$-space for which there is no $S$ such that $S^{-1}\pi(g)S$ is an isometry. Nevertheless, the previous theorem tells us that $\mathbb{Z}$ is $p$-isometrisable.

\begin{lemma} \label{induction}
Let $H$ be a subgroup of $\Gamma$. If $\Gamma$ is $(p,q)$-isometrisable, then $H$ is also $(p,q)$-isometrisable.
\end{lemma}

The classical construction of the induced representation (see for example Theorem 2.8 in \cite{Pisier}) remains valid for uniformly bounded representations on $p$-spaces with an only tiny change that instead of using a direct sum of Hilbert spaces one should use a $p$-sum of $p$-spaces, which is again a $p$-space.

\section{Our friend the free group}

In this section we prove our second main result.
\begin{theorem}\label{free}
For any $p\in (1,\infty)$ and any $r\geq 2$, the free group $\mathbb{F}_r$ on $r$ generators is not $(2,p)$-isometrisable. In particular, $\mathbb{F}_r$ is not $p$-isometrisable for any $p \in (1,\infty)$.

\end{theorem}
Firstly we need to prove the following lemma.
\begin{lemma}\label{estimate}
Let $A(\alpha)= \left(\begin{smallmatrix}
1 & i\alpha \\
i\alpha & 1 
\end{smallmatrix} \right) \in M_2(\mathbb{C})$, where $\alpha$ is some real number. Then for $p\in [1,2]$ we have
$$\|A\|_{p\to p} \leq 1+\theta|\alpha|+o(|\alpha|),$$ where $\theta=\frac{2}{p}-1$.
\end{lemma}
\begin{proof}
To estimate $\|A\|_{p\to p}$ we will use Riesz-Thorin interpolation theorem. For $p\in [1,2]$, we have 
$$\frac{1}{p}=\frac{\theta}{1}+\frac{1-\theta}{2}$$
with $\theta=\frac{2}{p}-1$. Thus, using the inequality $x^{\theta}y^{1-\theta}\leq \theta x+(1-\theta)y$, we obtain
$$\|A\|_{p\to p} \leq \|A\|^{\theta}_{1\to 1}\|A\|^{1-\theta}_{2\to 2}\leq\left(1+|\alpha|\right)^{\theta}\left(1+\frac{\alpha^2}{2}\right)^{1-\theta}\leq 1+\theta|\alpha|+(1-\theta)\frac{\alpha^2}{2}.$$
Here we used the fact that $\|A\|_{2\to 2} =\sqrt{1+\alpha^2}\leq 1+\frac{\alpha^2}{2}$.
\end{proof}

We denote by $B(x,r)$ the disk in $\mathbb C$ with center $x$ and radius $r$.
\begin{lemma}\label{spec}
Let $g_i$ for $ 1\leq i \leq r$ be isometries of a $p$-space $X$, then the spectrum of the operator $$\mu_1=\frac{1}{2r}\sum_{i=1}^r(g_i+g_i^{-1})$$ lies in the intersection $$B(0,1) \cap \{z \in \mathbb{C} \mid -\theta \leq \rm{Im(z)}\leq \theta\},$$ where $\theta=\frac{2}{p}-1.$
\end{lemma}
\begin{proof} First of all, it is clear that ${\rm sp}(\mu_1)$ lies in $B(0,1)$ since $\mu_1$ acts as a contraction.
Let us define a function $\Phi_{\alpha}\colon \mathbb{C} \to \mathbb{C}$ by the formula $\Phi_{\alpha}(z)=2r(1+i\alpha z)$ and consider $\Phi_{\alpha}(\mu_1)=2r+i\alpha\sum_{i=1}^r(g_i+g_i^{-1})$, where $\alpha$ is some real number.
Let us note that we can write this operator  in a following way:

$$\Phi_{\alpha}(\mu_1)=\begin{pmatrix}
1 & g_1 & \ldots & 1 & g_r\\
\end{pmatrix} \begin{pmatrix}
A(\alpha) & 0 & \dots& 0\\
0 & A(\alpha) & \dots & 0\\
\vdots & \vdots &\ddots & \vdots\\
0 & 0 & \dots & A(\alpha)
\end{pmatrix}\begin{pmatrix}
1\\
g_1^{-1}\\
\vdots\\
1\\
g_r^{-1}\\
\end{pmatrix}=x^tA_r(\alpha)y,$$
where $y\in L(X,\ell^p_{2r}(X))$ and $x^t\in L(\ell^p_{2r}(X),X)$.
Hence, 
$$\left\Vert \Phi_{\alpha}(\mu_1)\right\Vert \leq \|x^T\|\|A_r\|_{p \to p}\|y\|.$$
We can estimate the norm of $y\in L(X,\ell^p_{2r}(X))$ by $\|y\|\leq (2r)^{\frac{1}{p}}$. Indeed, for any $v\in X$ we have
$$\sqrt[p]{\sum_{i=1}^r(\|v\|^p+\|g_i^{-1}v\|^p)}\leq (2r)^{\frac{1}{p}}\|v\|.$$ 
By duality, we obtain $\|x^t\|\ \leq (2r)^{\frac{p-1}{p}}$.
Clearly, $\|A_r\|_{p\to p}=\|A(\alpha)\|_{p\to p}$ and hence by Lemma \ref{estimate} we have
$$\|A_r\|\leq 1+\theta|\alpha|+o(|\alpha|),$$
where $\theta=\frac{2}{p}-1$ if $p\in [1,2]$.
Putting everything together, we obtain
$$\Phi_{\alpha}({\rm sp}(\mu_1))={\rm sp}(\Phi_{\alpha}(\mu_1)) \subset B(0,2r(1+|\alpha|\theta+o(|\alpha|))).$$
Hence, for positive $\alpha$ we have
$${\rm sp}(\mu_1)\subset \Phi_{\alpha}^{-1}(B(0,2r(1+|\alpha|\theta+o(|\alpha|)))) \subset \left\{z\in \mathbb{C}\mid {\rm Im}(z) \geq -\left(\theta +\frac{o(|\alpha|)}{|\alpha|}\right)\right\}.$$
Since $\alpha$ can be chosen arbitrarily small we obtain
$${\rm sp} (\mu_1) \subset \{z\in \mathbb{C}\mid \rm{Im}(z) \geq -\theta\}.$$
By symmetry we get ${\rm sp}(\mu_1) \subset B(0,1) \cap \{z \in \mathbb{C} \mid -\theta \leq \rm{Im}(z)\leq \theta\}$ as claimed.
\end{proof}

In order to prove our main theorem, we will now make use of the family of non-unitarisable uniformly bounded representations of $\mathbb{F}_r$  constructed by Pytlik in \cite{Pytlik}. Let us briefly recall his construction. Pytlik constructs an analytic series of uniformly bounded representatinons $\pi_z$ of $\mathbb{F}_r$, defined through the action of $\mathbb{F}_r$ on its Poisson boundary $(\Omega,\mu)$ with respect to the canonical simple random walk. These representations $\pi_z\colon \mathbb{F}_r\to {\rm GL}(L^2(\Omega,\mu))$ are indexed by complex numbers from the ellipse
$$E=\left\{ z\in\mathbb{C} : \left\vert z-\frac{\sqrt{2r-1}}{r}\right\vert +\left\vert z+\frac{\sqrt{2r-1}}{r}\right\vert <2  \right\}.$$
Let us summarize  the properties of $\pi_z$.
\begin{theorem}[\cite{Pytlik}]\label{rep}
The representations $\pi_z, z\in E$ form an analytic family of uniformly bounded representations of $\Gamma$ on the Hilbert space $L^2(\Omega, \mu)$. Moreover:
\begin{enumerate}
\item Each $\pi_z$ is irreducible.
\item $\pi_z^*(x)=\pi_{\bar{z}}(x^{-1})$. In particular, $\pi_z$ is a unitary representation if $z$ is real.
\item The spectrum in $L^2(\Omega, \mu)$ of the operator $\pi_z(\mu_1)$ is contained in the set $$\{z\}\cup \left[-\frac{\sqrt{2r-1}}{r}, \frac{\sqrt{2r-1}}{r}\right].$$ 
\item The eigenspace corresponding to $\{z\}$ is one-dimensional and consists of constant functions.
\end{enumerate}
\end{theorem} 

Recall that by Lemma \ref{spaces} that $L^2(\Omega,\mu)$ is a $p$-space for any $p\in(1,\infty)$. We are now ready to finish the proof.

\begin{proof}[Proof of Theorem \ref{free}:]
Let $p \in (1,2]$, then $\theta<1$ and there exists such $r$ that $\theta < \frac{r-1}{r}$. Consider Pytlik's construction for this particular $r \in \mathbb N$. There exists $z_0$ such that $z_0\in E$, but 
$z_0 \notin \{z\in \mathbb{C} \mid -\theta \leq {\rm Im}(z) \leq \theta \}$.
Assume that $\pi_{z_0} \colon \mathbb F_r \to {\rm GL}(L^2(\Omega,\mu))$ is $p$-isometrisable, equivalently there exists a $p$-norm on $L^2(\Omega,\mu)$ such that $\pi_{z_0}(g)$ is isometric for any $g \in \mathbb{F}_r$. When this is the case, by Lemma \ref{spec}, we conclude that 
$${\rm sp}(\pi_{z_0}(\mu_1)) \subset \{z\in \mathbb{C} \mid -\theta \leq \rm{Im}(z) \leq \theta \}.$$
However, by the third property from Theorem \ref{rep} we have $z_0 \in {\rm sp}(\pi_{z_0}(\mu_1)).$ This is a contradiction. Thus, the group $\mathbb F_r$ is not $(2,p)$-isometrisable. We can now extend this to all groups containing a non-abelian free subgroup by Lemma \ref{induction}.
\end{proof}

\section{Conjectures and open questions} \label{final}

Let us recall the definition of the space of Littlewood functions $T_1(\Gamma)$.
\begin{defin}
The space of Littlewood functions is the space $T_1(\Gamma)$ of all the functions 
$f\colon \Gamma \to \mathbb{C}$
which admit the following decomposition: there exist functions $f_1\colon \Gamma\times \Gamma \to \mathbb{C}$ and $f_2\colon \Gamma\times \Gamma \to \mathbb{C}$ such that
$$f(s^{-1}t)=f_1(s,t)+f_2(s,t) \, \,\, \forall s, t \in \Gamma$$
and
$$\sup \limits_{s}\sum_{t }|f_1(s,t)| < \infty, \,\, \, \sup \limits_{t}\sum_{s}|f_2(s,t)| < \infty.$$
The space $T_1(\Gamma)$ is equipped with the norm
$$\|f\|_{T_1(\Gamma)}=\inf \left \lbrace\ \sup_s\sum_t|f_1(s,t)|+\sup_t\sum_s|f_2(s,t)|) \right \rbrace,$$
where infimum runs over all decompositions of this form.
\end{defin}
Properties of this space of functions are related to amenability and unitarisability of discrete groups. First of all, by the result of Wysocza\'nski \cite{Wysoczanski} amenability is characterized by the condition $T_1(\Gamma)\subseteq \ell^1(\Gamma)$. By the result of Bo\.zejko and Fendler \cite{BozejkoFendler} if $\Gamma$ is unitarisable then $T_1(\Gamma) \subseteq \ell^2(\Gamma)$. Since for any non-amenable group there exists some $\varepsilon>0$ such that $T_1(\Gamma)\nsubseteq \ell ^{1+\varepsilon}(\Gamma)$, see \cite{GGMT}, it is interesting to study the Littlewood exponent of a group $\Gamma$ defined as
$${\rm Lit}(\Gamma)=\inf\{p \mid T_1(\Gamma) \subset \ell^p(\Gamma)\}.$$
This quantity was studied \cite{GGMT} proving various results and giving a construction of a group with $1< {\rm Lit}(\Gamma)< \infty$.
It would be very interesting to understand which properties of a group $\Gamma$ imply that $T_1(\Gamma) \subset \ell^p(\Gamma).$ Maybe the following is true.
\begin{conjecture}\label{con} Let $1 \leq p \leq 2$.
Assume that $\Gamma$ is $(2,p)$-isometrisable. Then $T_1(\Gamma) \subset \ell^{q}(\Gamma)$, where $q$ is the conjugate exponent.
\end{conjecture}
For $p=2$ this is a classical result of Bo\.zejko and Fendler (see \cite{BozejkoFendler}). The proof of this theorem proceeds in two steps. Firstly, one needs to prove that if $\Gamma$ is unitarisable, then $T_1(\Gamma)\subset B(\Gamma)$, where $B(\Gamma)$ is a space of coefficients of unitary representations. Secondly, using the fact that $B(\Gamma)$ has cotype $2$, one can prove that $T_1(\Gamma)\subset \ell^2(\Gamma)$. The possible strategy to prove Conjecture \ref{con} is to use $B_p(\Gamma)$ instead of $B(\Gamma)$, where $B_p(\Gamma)$ is a space of coefficients of isometric representation on $q$-spaces, see \cite{Runde} for details:
$$B_p(\Gamma)=\left\{ f \colon \Gamma \to \mathbb{C} \mid f \textnormal{ is a coefficient of some } (\pi,E) \subset {\rm Rep}_{q}(\Gamma)\right\}.$$
The fist step of the proof is quite similar to the classical case, see for example \cite{BozejkoFendler} or \cite{Pisier}. Thus we have,
\begin{prop}
If $\Gamma$ is $(2,p)$-isometrisable, then $T_1(\Gamma)\subset B_{q}(\Gamma).$
\end{prop}
To obtain $T_1(\Gamma)\subset \ell^q(\Gamma)$ one would need to use the cotype argument. Maybe the following conjecture that appeared in \cite{daws} holds:
\begin{conjecture}[Daws]
The space $B_p(\Gamma)$ has cotype $q$, where $\frac{1}{p}+\frac{1}{q}=1$.
\end{conjecture}

\section*{Acknowledgements}
We thank Nicolas Monod, Tim de Laat, Hannes Thiel and Eusebio Gardella for helpful comments on a first version of this note. This research was supported in part by the ERC Consolidator Grant No.\ 681207. The results presented in this paper are part of the PhD project of the first author.


\begin{bibdiv}
\begin{biblist}
\bib{BozejkoFendler}{article}{
title={Herz-Schur multipliers and uniformly bounded representations of discrete groups},
  author={Marek Bo{\.z}ejko and  Gero Fendler},
  journal={Archiv der Mathematik},
  volume={57},
  number={3},
  pages={290--298},
  year={1991},
  publisher={Springer}
}
\bib{arch}{article}
{
  title={Isometric dilations and {$H^{\infty}$}-calculus for bounded analytic semigroups and {R}itt operators},
  author={C{\'e}dric Arhancet. Stephan Fackler. Christian Le Merdy},
  journal={Transactions of the American Mathematical Society},
  volume={369},
  number={10},
  pages={6899--6933},
  year={2017}
}


\bib{daws}{article}
{
  title={p-operator spaces and Figa-Talamanca--Herz algebras},
  author={Daws, Matthew},
  journal={Journal of Operator Theory},
  pages={47--83},
  year={2010},
  publisher={JSTOR}
}


\bib{Day}{article}
{
  title={Means for the bounded functions and ergodicity of the bounded representations of semi-groups},
  author={Day, Mahlon M.},
  journal={Transactions of the American Mathematical Society},
  volume={69},
  number={2},
  pages={276--291},
  year={1950},
  publisher={JSTOR}
}
\bib{Dixmier}{article}{
  title={Les moyennes invariantes dans les semi-groupes et leurs applications},
  author={Dixmier, Jacques},
  journal={Acta Sci. Math. Szeged},
  volume={12},
  number={Leopoldo Fej{\'e}r et Frederico Riesz LXX annos natis dedicatus, Pars A},
  pages={213--227},
  year={1950}
}
\bib{GGMT}{article}
{
  title={Asymptotics of Cheeger constants and unitarisability of groups},
  author={Maria Gerasimova. Dominik Gruber.Nicolas Monod. Andreas Thom},
  journal={arXiv preprint arXiv:1801.09600},
  year={2018}
}
\bib{Gillaspie}{article}
{
title={Power-bounded Invertible Operators and Invertible Isometries on $L^p$-Spaces},
  author={Gillespie, Alastair},
  booktitle={Operator Semigroups Meet Complex Analysis, Harmonic Analysis and Mathematical Physics},
  pages={241--252},
  year={2015},
  publisher={Springer}

}




\bib{Kwapien}{article}
{
  title={On operators factorizable through $ L_p$-space},
  author={Kwapien, Stanislaw},
  journal={M{\'e}moires de la Soci{\'e}t{\'e} Math{\'e}matique de France},
  volume={31},
  pages={215--225},
  year={1972}
}
\bib{Nakamura}{article}{
  title={Group representation and Banach limit},
  author={Masahiro Nakamura  and Zir{\^o} Takeda},
  journal={Tohoku Mathematical Journal, Second Series},
  volume={3},
  number={2},
  pages={132--135},
  year={1951},
  publisher={Mathematical Institute, Tohoku University}
}



\bib{Pisier}{book}{
 title={Similarity problems and completely bounded maps},
  author={Pisier, Gilles},
  year={2004},
  publisher={Springer}
}


\bib{Pytlik}{article}{
title={Spherical functions and uniformly bounded representations of free groups},
  author={ Pytlik, Tadeusz},
  journal={Studia. Math.},
  volume={100},
  number={3},
  pages={237-250},
  year={1991},
  }
\bib{Runde}{article}{
  title={Representations of locally compact groups on QSLp-spaces and a p-analog of the Fourier-Stieltjes algebra},
  author={Runde, Volker},
  journal={Pacific J. Math},
  volume={221},
  pages={379--397},
  year={2005}
}
\bib{Wysoczanski}{article}{
    AUTHOR = {Wysocza\'nski, Janusz},
     TITLE = {Characterization of amenable groups and the {L}ittlewood
              functions on free groups},
   JOURNAL = {Colloq. Math.},
  FJOURNAL = {Colloquium Mathematicum},
    VOLUME = {55},
      YEAR = {1988},
    NUMBER = {2},
     PAGES = {261--265},
      ISSN = {0010-1354},
   MRCLASS = {43A15 (22D05)},
  MRNUMBER = {978923},
MRREVIEWER = {Michael Leinert},
       URL = {https://doi.org/10.4064/cm-55-2-261-265},
}



\end{biblist}
\end{bibdiv}
\end{document}